\documentclass[12pt,reqno]{amsart}
\usepackage{amssymb, amsmath, amsthm}
\usepackage[backref]{hyperref}
\usepackage{fullpage, verbatim}

\newtheorem{lem}{Lemma}[section]
\newtheorem{theorem}[lem]{Theorem}

\newtheorem{prop}[lem]{Proposition}
\newtheorem{cor}[lem]{Corollary}

\newtheorem{claim*}{Claim}
\newtheorem{thm*}{Theorem}
\newtheorem{cor*}{Corollary}
\newtheorem{rmk}[lem]{Remark}

\newtheorem{thm}[lem]{Theorem}
\newtheorem{defn}[lem]{Definition}

\newtheorem{question}[lem]{Question}

\numberwithin{equation}{section}
\numberwithin{table}{section}

\def\p{\mathfrak{p}}
\def\q{\mathfrak{q}}

\def\Fpt{\mathbb{F}_p \left(t \right)}

\def\Gal{\text{Gal}}

\def\Fx{F \left( x \right)}
\def\vp{\varphi}
\def\vpx{\vp\left( x \right)}
\def\vpn{\vp^n}
\def\vpnx{\vpn\left( x \right)}
\def\a{\alpha}
\def\b{\beta}

\def\P1{\mathbb{P}^1}
\def\Eu{E^{\text{sep}}}
\def\oo{\mathfrak{o}}
\def\ff{\mathfrak{f}}

\begin{document}

\baselineskip=17pt

\title[A Hasse Principle for Periodic Points]{A Hasse Principle for Periodic Points}

\author{Adam Towsley}

\address{
Adam Towsley\\
Department of Mathematics\\
CUNY Graduate Center\\
365 5th Avenue\\
New York, NY 10016-4309\\
}

\email{atowsley@gc.cuny.edu}

\date{March 6, 2013}

\begin{abstract}
Let $F$ be a global field, let $\vp \in \Fx$ be a rational map of degree at least 2, and let $\a  \in F$. 
We say that $\a $ is periodic if $\vpn \left( \a \right) = \a$ for some $n \geq 1$. 
A Hasse principle is the idea, or hope, that a phenomenon which happens everywhere locally should happen globally as well. 
The principle is well known to be true in some situations and false in others. 
We show that a Hasse principle holds for periodic points, 
and further show that it is sufficient to know that $\a$ is periodic on residue fields for every prime in a set of natural density density 1 to know that $\a$ is periodic in $F$.
\end{abstract}



\maketitle

\section{Introduction}\label{sec:Introduction}

Let $F$ be a \textbf{global field}, by which we mean a finite extension of $\mathbb{Q}$ or $\mathbb{F}_p \left( t \right)$. 
Let $\vp: \mathbb{P}^1\left(F \right) \rightarrow \mathbb{P}^1\left(F \right)$ be a rational map. 
For any integer $n \geq 0$, write $\vpnx = \vp \circ \cdots \circ \vpx$ for the $n$th iterate of $\vpx$ under composition. 
The \textit{forward orbit} of a point $\a \in F$ is defined to be $\mathcal{O}_\vp \left( \a \right) = \left\lbrace \a, \vp \left( \a \right), \vp^2 \left( \a \right), \dots \right\rbrace$. 
Similarly the \textit{strict forward orbit} of $\a$ is defined to be $\mathcal{O}^+_\vp \left( \a \right) = \left\lbrace \vp \left( \a \right), \vp^2 \left( \a \right), \dots \right\rbrace$ and 
the \textit{back orbit} of $\a$ is defined to be the set $\mathcal{O}_\vp^- \left( \a \right) = \left\lbrace  \b \in \mathbb{P}^1 \left(\overline{F}\right) : \vpn \left( \b \right) = \a \text{ for some }n \geq 1 \right\rbrace$. A point $\a $ is said to be \textit{periodic} if $\vpn \left( \a \right) = \a$ for some $n$, and more generally $\a$ is said to be \textit{preperiodic} if its forward orbit is finite. If $\a$ has an infinite forward orbit we say that $\a$ is a \textit{wandering point}. If the back orbit of $\a$ is finite then we say that $\a$ is an \textit{exceptional point}. 

For a prime ideal $\p $ there is a well defined `reduction mod $\p$' map 
$r_\p: \mathbb{P}^1\left(F \right) \rightarrow \mathbb{P}^1 \left( \ff_\p \right)$ where $\ff_\p$ is the residue field of $F$ at $\p$. 
For any $\a \in F$ denote $r_\p \left( \a \right)$ by $\overline{\a} \in \ff_\p$. 
A rational function $\vp \in \Fx$ can be written as $\vp \left( x,y \right) = \left[\vp_1 \left(x,y \right) ; \vp_2 \left( x,y \right) \right]$ 
where $\vp_1$ and $\vp_2$ are homogeneous polynomials. 
Denote by $\overline{\vp}$ the reduction of $\vp$ modulo $\p$, found by reducing the coefficients of $\vp_1$ and $\vp_2$ modulo $\p$. 
For all but finitely many primes $\overline{\vp}: \mathbb{P}^1 \left( \ff_\p \right) \rightarrow \mathbb{P}^1 \left( \ff_\p \right)$ is a morphism. 
Write $M_F^0$ for the set of finite places of $F$. 
Each $v \in M_F^0$ has a canonical identification with a unique prime ideal $\p_v \subseteq \oo_F$, the converse is true as well.
The use of $v$ and $\p_v$ will be interchangably throughout the paper. 

For each prime $\p$ the residue field $\ff_\p$ is finite. 
Therefore any $\overline{\a} \in \ff_\p$ must be either periodic or strictly preperiodic under $\overline{\vp} \in \Fx$. 
If $\overline{\a} \in \ff_\p$ is periodic we will say that $\a$ has \textit{periodic reduction} at $\p$. 
Note that we could also define the reduced orbit of $\a$ modulo $\p$ by taking the reduction of the ordered set of points 
$\mathcal{O}_\vp \left( \a \right)$ modulo $\p$. 
This definition corresponds to $\mathcal{O}_{\overline{\vp}} \left( \overline{\a } \right)$ if $\p$ is a prime of good reduction for $\vp$ 
and is well defined if $\p$ is a prime of bad reduction for $\vp$. 
If $\p$ is a prime of bad reduction for $\vp$ we will say $\a$ has periodic reduction at $\p$ 
if its reduced orbit is periodic under the second definition. 

Given any point $\a \in F$ one might ask if it is possible to determine if $\a$ is periodic in $F$ based on its reduction modulo $\p$ for various $\p$? 
In other words, can local information about periodicity give global information? 
To that end we prove the following theorem, a corollary of which can be thought of as a Hasse principle for periodic points.

\begin{thm*}\label{hasse}
For a global field $F$, rational map $\vp \in \Fx$ of degree at least 2, and point $\a \in F$ the following are equivalent:
\begin{itemize}
\item[(i)] $\a$ is periodic.
\item[(ii)] $\a$ has periodic reduction for every prime $\p$. 
\item[(iii)] $\a$ has periodic reduction for every prime $\p \in \mathcal{P}$, where $\mathcal{P}$ is a set of primes with natural density 1. 
\end{itemize}
\end{thm*}

\begin{cor*}[A Hasse Principle for Periodic Points]\label{Cor:hasse}
 Given the same setup as the theorem, $\a$ is periodic in $F$ if and only if it is periodic in $F_\p$ (the completion of $F$ at $\p$) for every $\p$.
\end{cor*}
\begin{proof}
 The proof is trivial. If $\a$ is periodic in $F$ then it is periodic in $F_\p$ for each $\p$. 
 If $\a$ is periodic in $F_\p$ for every prime $\p$, then it must be periodic in the residue field of $F_\p$ so $\a$ must be periodic in $F$ by the theorem.
\end{proof}

Clearly $(i) \Rightarrow (ii) \Rightarrow (iii)$ so the only work involved is proving $(iii) \Rightarrow (i)$. 
To do this for number fields we use the following theorem of Benedetto, Ghioca, Hutz, Kurlberg, Scanlon, and Tucker from \cite{BGHKST}.

\begin{thm}[Benedetto, Ghioca, Hutz, Kurlberg, Scanlon, and Tucker]\label{nfintersection}
If $F$ is a number field, $\vp \in \Fx$ a rational map of degree at least 2, and $\a, \b \in F$ are points such that $\b \notin \mathcal{O}_\vp^+ \left( \a \right)$ then there is a set of primes $\mathcal{P}$ with positive density such that $\overline{\b} \notin \mathcal{O}^+_{\overline{\vp}} \left( \overline{\a} \right)$ for every $\p \in \mathcal{P}$.
\end{thm}

To complete the proof of Theorem \ref{hasse} when $F$ is a function field we prove Theorem \ref{ffintersection}, the analog of Theorem \ref{nfintersection}. A key element of Theorem \ref{nfintersection} is that if $\a$ wanders then it will have non-periodic reduction for infinitely many primes. For a number field $F$ this follows easily from the following theorem of Silverman which appears in \cite{Silverman93}.

\begin{theorem}[Silverman]\label{Silverman}
If $F$ is a number field, $\vpx \in \Fx$ is a rational map of degree at least 2, and $\a, \b \in F$ are points such that $\b$ is not exceptional then only finitely many elements of $\mathcal{O}_\vp \left( \a \right)$ are $S$-integral to $\b$. 
\end{theorem}

We cannot use Silverman's Theorem when $F$ is a function field since the proof requires Roth's Theorem, which is false for fields with positive characteristic (see section 6.2 in \cite{Bombieri}).  Thus we prove the following similar theorem.

\begin{thm*}\label{integrality}
If $F$ is a global field, $\vp \in \Fx$ is a rational map of degree at least 2, and $\a, \b \in \mathbb{P}^1 \left( F \right)$ are points such that $\b$ is periodic and not exceptional then only finitely many elements of $\mathcal{O}_\vp \left( \a\right) $ are $S$-integral to $\b$. 
\end{thm*}

As a corollary to Theorem \ref{integrality} we are able to prove that if $\a \in F$ is not periodic then there are infinitely many primes $\p$ for which $\a$ has non-periodic reduction. To prove Theorem \ref{integrality} we will first prove Lemma \ref{Runge_proof} using Runge's method. Runge developed the method in the 1880's in \cite{Runge}. In 1983, almost 100 years later, it was generalized by Bombieri in \cite{BombieriRunge}. More recently Runge's method has been used by Levin in \cite{Levin} and in a dynamical setting by Corvaja, Sookdeo, Tucker, and Zannier in \cite{CSTZ}. We will then apply Theorem \ref{integrality} to recover several necessary dynamical results.

In Section \ref{sec:Prelims} we give a definition of global fields and describe the important Northcott property, 
we also develop the notion of integrality. 
In Section \ref{sec:Integrality} we prove a finiteness lemma which we apply to a dynamical setting in order 
to prove Theorem \ref{integrality}. 
In Section \ref{sec:Interesect} we then apply Theorem \ref{integrality} to prove Theorem \ref{ffintersection} 
which is the function field analog of Theorem \ref{nfintersection}. 
Finally, in Section \ref{sec:Hasse} we use Theorems \ref{nfintersection} and \ref{ffintersection} to prove Theorem \ref{hasse}. 
Having proved Theorem \ref{hasse} we discuss the possibility of strengthening it 
and the necessary conditions which would need to be applied. 

\textit{Acknowledgments}: The author would like to thank Tom Tucker and Xander Faber for their help preparing this paper.

\section{Preliminaries}\label{sec:Prelims}
We now develop some of the necessary preliminary materials for proving Theorem. \ref{integrality}
\subsection{Global Fields and the Northcott Property}
In the introduction we made the following definition.
\begin{defn}\label{globalfield}
If $F$ is a finite extension of $\mathbb{Q}$ or $\Fpt$ then we say that $F$ is a \textbf{global field}.
\end{defn}
These types of fields are distinguished for two reasons; they have finite residue fields and they admit height functions which
have the Northcott property which we will state below as a theorem.

A global field $F$ is equipped with a set $M_F$ of inequivalent non-trivial absolute values $\left\vert \cdot \right\vert_v$ which satisfy the following two properties for any $\a \in F\setminus \left\lbrace 0 \right\rbrace$.
\begin{enumerate}
 \item $\left\vert \a \right\vert_v =1$ for all but finitely many $v \in M_F$, and
 \item $\displaystyle\prod_{v \in M_F} \left\vert \a \right\vert_v = 1 $. (The product formula). 
\end{enumerate}
Let $F$ be a global field and $\overline{F}$ be a fixed algebraic closure of $F$. 
The absolute logarithmic global height function $h:\mathbb{P}^1\left(\overline{F} \right) \rightarrow \mathbb{R}_{\geq 0 }$ is defined in the standard way.
Let $\a = \left[ \a_0 : \a_1 \right] \in \mathbb{P}^1 \left( \overline{F} \right)$ be an element of the finite extension $F^\prime / F$. 
The height of $\a$ is given by 
$$ h \left( \a \right) = \sum_{w \in M_{F^\prime}} \log \max \left\lbrace \left\vert \a_0 \right\vert_w, \left\vert \a_1 \right\vert_2 \right\rbrace.$$
Because $F$ satisfies the product formula $h \left( \a \right)$ is defined independently of the choice of coordinates of $\a$. 
The normalization of elements of $M_{F^\prime}$ are arranged so that $h\left( \a \right)$ is independent of the choice of field $F^\prime$ containing $\a$.
These definitions follow the notation and conventions of \cite[\S1]{Bombieri}, which is an excellent source for a reader looking for more details. 

As mentioned above, heights on global fields satisfy the Northcott property which we state formally here. 
\begin{thm}[The Northcott Property]\label{Nortcott}
If $F$ is a global field, and $N$ is any positive real number then there are only finitely many $\a \in F$ such that $h \left( \a \right) < N$. 
\end{thm} 

\begin{proof}
If $F$ is a number field this is well known, for example see \cite{SilvermanDynamicsBook}. If $F$ is a function field with a finite field of constants then the Northcott property follows from the fact that there are only finitely many elements $\a \in F$ of bounded degree. 
\end{proof}

\subsection{$\left(D,S \right)$-Integers}\label{subsec:aSintegers}
Let $F$ be a global field, and let $S$ be a finite set of places of $F$ containing the infinite places. 
For any place $v$ of $F$ define the chordal metric on $\P1 \left( \mathbb{C}_v \right)$ as follows:
$$\delta_v \left( P,Q \right) = \frac{\left\vert x_1 y_2 - x_2 y_1 \right\vert_v}{\max \left\lbrace \left\vert x_1 \right\vert_v, \left\vert y_1\right\vert_v \right\rbrace \max \left\lbrace \left\vert x_2 \right\vert_v, \left\vert y_2\right\vert_v \right\rbrace},$$
where $P = \left[ x_1 :  y_1 \right]$, $Q = \left[x_2, y_2 \right]$. 
Note that $0 \leq \delta_v \left( \cdot, \cdot \right) \leq 1$. 
The ring of $S$-integers of $F$ is the set $\left\lbrace b \in F : \delta_v \left( b, \infty \right) = 1 \text{ for every } v \notin S\right\rbrace$. 
We can generalize this by replacing $\infty$ with an arbitrary point $a \in F$ and create the set of $\left(a, S \right)$-integers, $\left\lbrace b \in F : \delta_v \left( b, a \right) = 1 \text{ for every } v \notin S \right\rbrace$.
This can further be generalized by replacing the point $a$ with a finite set of points $D \subseteq \P1 \left( F \right)$ to create the $\left(D, S \right)$-integers, $\left\lbrace b \in F : \delta_v \left( b, d \right) = 1 \text{ for every } d \in D \text{ and every }v \notin S \right\rbrace$.
This interpretation of $\left(D,S\right)$-integers allows us to generalize to an algebraic closure $\overline{F}$ of $F$ by varying over all embeddings of $b$ and the points of $D$ into $\overline{F}$. 

Intuitively, $b $ is a $\left( D, S \right)$-integer if it avoids the elements of $D$ modulo $v$ for all places $v \notin S$. 
We now prove the following lemma which will allow us to compute the height of $\left(D,S\right)$-integers with a finite sum in Lemma \ref{Runge_proof}.

\begin{lem}\label{S_Int}
 Let $F$ be a global field, $f \in \Fx$ an irreducible polynomial, $D \subseteq \overline{F}$ the set of roots of $f$,
 $S$ a finite set of places containing the infinite places and any place which is a pole of a coefficient of $f$ and let $b \in F$ be $S$-integral to $D$.
 If $\left\vert f \left( b \right) \right\vert_v < 1 $ then $v \in S$. 
\end{lem}

\begin{proof}
Let $G$ be the Galois group of $f$ and let $d \in D$ be a root of $f$. 
For some non-zero $c \in F$ we can write $f \left( b \right) = c \cdot  \displaystyle\prod_{\sigma \in G} \left( b- \sigma \left( d \right) \right)$.
If $\left\vert f \left( b \right) \right\vert_v < 1 $ then there exists a $\sigma \in G$ such that $\left\vert b - \sigma \left( d \right) \right\vert_v <1$.

Assume that $b \neq \infty$. We write $b$ and $\sigma \left( d \right)$ projectively as $b = \left[ b; 1 \right]$ and $\sigma \left( d \right) = \left[ \sigma \left( d \right); 1 \right]$. 
By definition of the chordal metric we see that 
$$\delta_v \left( b, \sigma \left( d \right) \right) = \frac{\left\vert b - \sigma \left( d \right) \right\vert_v}{\max \left\lbrace \left\vert b \right\vert_v, \left\vert 1 \right\vert_v \right\rbrace \max \left\lbrace \left\vert \sigma \left( d \right) \right\vert_v, \left\vert 1 \right\vert_v \right\rbrace }  \leq \left\vert b - \sigma \left( d \right) \right\vert_v <1. $$
Since $b$ is a $\left( D,S \right)$-integer it follows that $v \in S$. 

Now assume that $b = \infty = \left[0;1 \right]$ and $\sigma \left( d \right) = \left[ \sigma \left( d \right); 1 \right]$. Then
$$ \delta_v \left( b, \sigma \left( d \right) \right) = \frac{\left\vert b - \sigma \left( d \right) \right\vert_v}{\max \left\lbrace \left\vert 1 \right\vert_v, \left\vert 0 \right\vert_v \right\rbrace \max \left\lbrace \left\vert \sigma \left( d \right) \right\vert_v, \left\vert 1 \right\vert_v \right\rbrace }  \leq \left\vert b - \sigma \left( d \right) \right\vert_v <1. $$
Since $b$ is a $\left( D,S \right)$-integer it follows that $v \in S$. 
\end{proof}

\section{Integrality}\label{sec:Integrality}

\subsection{Runge's Method}\label{sec:Runge}
In this section we will prove a finiteness result which is critical to Theorem \ref{integrality}. To begin we prove the following bound.

\begin{prop}\label{bound}
Let $F$ be a global field, $S$ be a finite set of places, and $f_1, \dots, f_t \in F \left[x \right]$ be distinct irreducible polynomials with leading coefficients $a_1, \dots, a_t$. Let $v \in S$ and let $C_v = \left[ \frac{1}{2} \right]^d \cdot \left[ \displaystyle\prod_i \min \left\lbrace \left\vert a_i \right\vert_v, 1 \right\rbrace  \right] \cdot \left[ \displaystyle\prod_{} \min \left\lbrace \left\vert \b - \gamma \right\vert_v, 1 \right\rbrace  \right] $ where $\beta$ and $\gamma$ are taken to range over roots of all pairs polynomials $f_i \neq f_j$ and $d$ is the maximal degree of the $f_j$. If $\alpha \in F$ then for all but at most one $f_i$ we have that $\left\vert f_j \left( \alpha \right) \right\vert_v \geq C_v$. 
\end{prop}
\proof{
For any $v \in S$ fix an embedding $F \hookrightarrow \mathbb{C}_v$. Pick a root $\beta_i$ of some $f_i$ such that $\left\vert \alpha - \beta_i \right\vert_v \leq \left\vert \alpha - \beta \right\vert_v$ for each root $\beta$ of of any $f_j$. For any $j \neq i$ we have that 
\begin{eqnarray*}
\left\vert f_j \left( \a \right) \right\vert_v & = & \left\vert a_j \right\vert_v \cdot \prod_{ f_j \left( \gamma \right) = 0 } \left\vert \a - \gamma \right\vert_v \\
& \geq & \left\vert a_j \right\vert_v \cdot \prod_{ f_j \left( \gamma \right) = 0 } \frac{1}{2} \left\vert \b_i - \gamma \right\vert_v \ \ \left(\text{ since }\left\vert \alpha - \beta_i \right\vert_v \text{ is minimal}\right)\\
& \geq & \left[ \frac{1}{2} \right]^d \cdot \left[ \displaystyle\prod_i \min \left\lbrace \left\vert a_i \right\vert_v, 1 \right\rbrace  \right] \cdot \left[ \displaystyle\prod_{} \min \left\lbrace \left\vert \b - \gamma \right\vert_v, 1 \right\rbrace  \right] \\
& = & C_v.
\end{eqnarray*}
}

We now use Proposition \ref{bound} to show that if a rational function has enough irreducible factors in relation to a set of primes $S$, then only finitely many points can be $S$-integral to its vanishing divisor.

\begin{lem}\label{Runge_proof}
Let $F$ be a global field, and $S$ be a finite set of places (containing the infinite places) with $\left\vert S \right\vert = r$. 
If $\vp \in \Fx$ is a rational map where the numerator factors into irreducible polynomials $f_1, \dots, f_t $ with $t>r$ and 
$D = \left\lbrace \delta \in F \vert \vp \left(\delta \right) = 0 \right\rbrace$ is the vanishing set of $\vp$ then there are at most finitely many points $\a \in F$ which are $S$-integral to $D$.
\end{lem}
\proof{By Proposition \ref{bound}, for each $v \in S$ there is an effectively computable $C_v$ such that for any $\a \in F$ we have that 
$\left\vert f_i \left( \a \right) \right\vert_v  \geq C_v$ for all but at most one $f_i$. 
Since there are fewer places $v \in S$ than there are polynomials $f_i$ it follows that there is an $f_k$ such that 
$\left\vert f_k \left( \a \right) \right\vert_v \geq C_v$ for every $v \in S$.

If $\a$ is $S$-integral to $D$ then it follows from Lemma \ref{S_Int} and the product formula that 
$$h \left( f_k \left( \a \right) \right) = \displaystyle\sum_{v \in S} - n_v \log^- \left\vert f_k \left( \a \right) \right\vert_v 
\leq \displaystyle\sum_{v \in S} - n_v \log C_v.$$ 
Where $n_v$ is the local degree of $v$. 
Since $h \left( f_k \left( \a \right) \right) \geq d_k h \left( \a \right) + M_k$, where $d_k = \deg\left( f_k \right)$ and 
$M_k$ is an effectively computable constant which depends only upon $f_k$ and not on $\a$, 
it follows that there is an effectively computable $C$ such that $$h \left( \a \right) \leq C.$$ Thus the set of such $\a$ must be finite by Northcott.
\
}

\subsection{Integrality in Dynamics}\label{subsec_dynamicintegrality}
We are now prepared to prove Theorem \ref{integrality}.

\begin{proof}[Proof of Theorem \ref{integrality}]
If $\a$ is preperiodic then $\mathcal{O}_\vp \left( \a \right)$ is finite and the result is trivial. So assume that $\a$ wanders.

Let $S$ be a finite set of primes which includes the archimidean primes and the primes of bad reduction. Additionally, let $\b$ have period $k$. As $\b$ is not exceptional, there exists an element $\b_1 \in \vp^{-k} \left( \b \right)\subseteq \overline{F}$ which is not in the periodic cycle of $\b$. Thus the numerator of $\vp^{k}\left( x\right) -\b$ contains an irreducible polynomial factor, $g_1 \in F \left[ x \right]$, of which $\b_1$ is a root. Similarly there is a $\b_2 \in \vp^{-k}\left(\b_1 \right)$ which does not lie in the periodic cycle of $\b$. As before we see that the numerator of $\vp^{2k} \left( x \right) - \b$ must contain an irreducible polynomial factor, $g_2 \in F \left[ x \right]$, of which $\b_2$ is a root. Note that $\b_2$ is not a root of $g_1$ and that $\b_1$ is not a root of $g_2$. Proceeding inductively we are able to produce distinct irreducible polynomials $g_1, \dots, g_{r+1} \in F \left[ x \right]$ where each $g_i$ divides the numerator of $\vp^{ik}\left( x \right) - \b$. As $\b$ is periodic with period $k$ it follows that all of the $g_i$ are factors of the numerator of $\vp^{k \left( r+1 \right)} \left( x\right)- \b$. We thus have a rational function, $\vp^{k \left( r+1 \right)} \left( x\right) - \b \in \Fx$, where $r+1$ distinct irreducible polynomials divide the numerator.

Let $D\subseteq \overline{F}$ be the roots of the numerator of $\vp^{k \left( r+1 \right)} \left( x\right) - \b$, note that since $\b$ is a root of $\vp^{k \left( r+1 \right)} \left( x\right) - \b $ it is an element of $D$. Lemma \ref{Runge_proof} shows that there are only finitely many elements $\a \in F$ which are $S$-integral to $D$. 
Since $\vp^n \left( \alpha \right)$ is $S$-integral to $\beta$ if and only if $\vp^{n-r-1} \left( \a \right)$ is $ S$-integral to $D$ it follows that there can be only finitely elements of $\mathcal{O}_\varphi \left( \alpha \right)$ which are $S$-integral to $\b$.
\end{proof}

Using the same notation we deduce the following corollaries.

\begin{cor}\label{Runge}
If $\a \in F$ is a wandering point and $\b \in F$ is a non-exceptional periodic point then there exists infinitely many prime ideals $\mathfrak{p}$ such that $\vpn\left( \a \right) \equiv \b \mod \mathfrak{p}$ for some $n$. 
\end{cor}
\begin{proof}

This is immediate from the definition of $S$-integrality. Since only finitely many elements of $\mathcal{O}_\vp \left( \a\right)$ are $S$-integral to $\b$ for any finite set of primes $S$ it follows that there must be infinitely many primes $\p$ for which $\vpn \left( \a \right) \equiv \b \mod \p$ for some $n$. 
\end{proof}

\begin{cor}\label{notperiodic}
Let $\vp \left(x \right) \in F \left( x \right) $, and $\a \in F$ be non-periodic. There exists infinitely many primes $\p$ for which $\a$ is not periodic modulo $\p$.
\end{cor}
\begin{proof}
If $\a$ is preperiodic then its strict forward orbit $\mathcal{O}^+_\vp \left( \a \right)$ is finite, denote it as $\left\lbrace \a_1, \dots, \a_l \right\rbrace$. For any $1 \leq i \leq l$ only finitely many primes contain $\left( \a - \a_i \right)$, since there are only finitely many $i$ we see that $\a$ can only have periodic reduction for finitely many primes. 

Assume now that $\a$ wanders, after possibly passing to an extension $E/F$ we can find a periodic $\b_1 \in E$ which is not exceptional. Let $\b_1$ have exact period $k$ and let $\mathcal{O}_\vp\left(\b_1 \right) = \left\lbrace \b_1, \b_2, \dots, \b_k\right\rbrace$. By Corollary \ref{Runge} there are infinitely many primes $\p \subseteq E$ such that $\vpn \left( \a \right) \equiv \b_1 \mod \p$. For each $\b_i$ there are only finitely many prime ideals which contain $\left(\a - \b_i \right)$, so there are only finitely many primes for which $\a \equiv \b_i \mod \p$. Thus there are only finitely many primes for which $\a$ is in the periodic cycle of $\b_1$ modulo $\p$ and is thus periodic modulo $\p$. Since there are infinitely many prime for which $\vpn \left( \a \right) \equiv \b_1 \mod \p $ for some $n$ we have that $n > k$ infinitely often and therefore $\a$ must be strictly preperiodic infinitely often.
\end{proof}

\section{Intersections in Orbits}\label{sec:Interesect}
We will now prove the following analog to Theorem \ref{nfintersection}.
\begin{thm*}\label{ffintersection}
If $F$ is a function field with a finite field of constants, 
$\vpx \in \Fx$ a rational map with separability degree of at least 2, 
and $\a,\b \in F$ are points such that $\b \notin \mathcal{O}_\vp^+ \left( \a \right)$, 
then there is a set of primes $\mathcal{P}$ with positive natural density such that 
$\overline{\b} \not\in \mathcal{O}_{\overline{\vp}}^+ \left( \overline{\a} \right)$ 
for any prime $\p \in \mathcal{P}$.
\end{thm*}

\subsection{Separability, Ramification, and Density}\label{subsec:geosep}
If $F$ is a function field with a finite field of constants then $F$ has positive characteristic. 
To begin proving Theorem \ref{ffintersection} we must first account for the fact that not all finite extensions $L/F$ are separable. 
Once these issues are dealt with we use Corollary \ref{notperiodic} and reproduce the proof from \cite{BGHKST}. 
We begin by making the following definitions.

\begin{defn}\label{GeoSep}
Let $F$ be a function field with a finite field of constants, $\vp \in \Fx$, $\vp = \frac{f\left( x \right)}{g \left( x \right)}$, and 
$\deg \left( \vp \right) = \max \left\lbrace \deg \left( f \right) , \deg \left( g  \right) \right\rbrace \geq 2$. 
We say that $\vp$ is \textbf{separable} (resp. inseparable, purely inseparable) if it induces a separable (resp. inseparable, purely inseparable) field extension. 
\end{defn}

In addition to a rational function being separable we will require that it have separable reduction. 
A map $\vp \in \Fx$ has \textbf{separable reduction} at a prime $\p$ if the reduction of $\vp$ modulo $\p$ is separable.
We now show that having separable reduction is not a strong condition to impose.

\begin{prop}\label{finite_sep_reduction}
If $\vp \in \Fx$ is separable with degree at least 2, then $\vp$ has separable reduction for all but finitely many primes $\p$ of $F$.
\end{prop}

\begin{proof}
Let $F$ have characteristic $p$ and write $\vp = f/g$ for two relatively prime polynomials $f,g \in F \left[ x \right]$. 
Since $\vp$ is separable there exists a term $a x^l$ of either $f$ or $g$ for which $a \neq 0 $ and $p \nmid l$. 
Let $S$ be the set of primes of $F$ which divide the numerator or denominator of $a \cdot \text{Res} \left( \vp \right)$. 
Then $\vp $ has separable good reduction for every $\p \notin S$. 
\end{proof}

If $\vp \in \Fx$ is an inseparable rational function then it can be written as $\vp \left( x \right) = h \left( x^r \right)$ 
where $h$ is separable and $r$ is the inseparable degree of $\vp$. 
Let $E/F$ be the splitting field of $\vp$ and define $\Eu$ to be the maximal separable sub-extension, that is $F \subseteq \Eu \subseteq E$. 
If $r \geq 2$ then $E/\Eu$ is a purely inseparable extension so any prime $\p$ of $F$ will ramify in $E$. 
Because of this fact we will modify our definition of a ramified prime.

\begin{defn}\label{ram_prime}
For $E $ and $\Eu$ as above we say that a prime $\p$ of $F$ ramifies in $E$ if it is ramified in $\Eu$. 
\end{defn}

Recall that the \textbf{natural density} of a set of places $S$ of a global field $F$ is defined to be 
$D \left( S \right) = \displaystyle\lim_{N \rightarrow \infty} \dfrac{\# \left\lbrace \text{places } 
v \in S \vert N_v \leq N \right\rbrace}{\# \left\lbrace \text{places } v \in F \vert N_v \leq N \right\rbrace}$, 
where $NS_v$ is the size of the residue field of $v$. 
We will prove our density result for an extension $L/F$, however when sets of primes with positive density in a field 
$L$ are intersected with a subfield $F$ they will yield a set of primes with positive density in the subfield, though the density may decrease. We prove this fact now.

\begin{lem}\label{densityextension}
Let $E$ be be a finite extension of any global field $F$ with $\left[E:F \right] = d$. Let $S_E$ be a set of primes of $E$ with positive density $\delta$ and let $S_F = \left\lbrace \p = \q \bigcap F \vert \q \in S_E \right\rbrace$. $S_F$ has positive density $\gamma $ where $\gamma \geq \frac{\delta}{d} > 0$.
\end{lem}

\begin{proof}
For each prime $\p$ of $F$ there are at most $d$ primes $\q_1, \q_2, \dots, \q_n$ of $E$ such that $q_i \bigcap F = \p $. Given any $N$, for every $d$ primes $v \in S_E$ with $N_v \leq N$ there is at least one prime $\p \in S_F$ with $N_p \leq N$. Taking the limit as $N \rightarrow \infty$ we see that the density of $S_F$ is at least $\frac{\delta}{d}> 0$. 

\end{proof}
To obtain our density result we adapt the arguments of \cite{BGHKST} to prove lemmas 
\ref{RamPoints}, \ref{Ramification}, and \ref{big} in the function field setting. 
We will use Chebotarev's Density Theorem for function fields as proved by Murty and Scherk in \cite{MurtyScherk} which states the following.

\begin{thm}[Chebotarev's Density Theorem for Function Fields] \label{MurtyScherkCheb}
Let $F$ be a function field with a finite field of constants, and $E/F$ be a Galois extension with Galois group $G$. If $C \subset G$ is a conjugacy class, $\psi$ is the number of unramified primes of $E$, $\psi_C$ is the number of unramified primes of $E$ whose Frobenius substitution corresponds to $C$ then 
$$\left\vert \psi_C - \psi  \frac{\left\vert C \right\vert}{\left\vert G \right\vert} \right\vert \leq \frac{B}{\sqrt{N}} ,$$ 
where $N$ is the size of the constant field of $F$, and $B$ is a constant depending on the genus of the curve associated to $E$, the ramification of $E/F$, $\left\vert C \right\vert$, and $\left\vert G \right\vert $. 
\end{thm} 

\subsection{Proof of Theorem \ref{ffintersection}}

To use Theorem \ref{MurtyScherkCheb} it is necessary to have an unramified field extension, 
thus we prove the following lemmas relating critical points of $\vp$ to ramified primes in a field extension of $F$ 
in the sense of Definition \ref{ram_prime}.

\begin{lem}\label{critical_point}
Let $k$ be a complete and algebraically closed non-Archimidean field with nontrivial valuation, 
also let $\vp \in k \left( x\right)$ be a rational function of degree at least 2 with good and separable reduction.
If $\b \in \mathbb{P}^1 \left( k \right)$ is such that the reduction $\overline{\vp}$ has a critical point at $\overline{\b}$,
then there is a critical point for $\vp$ lying in the same residue disk as $\b$. 
\end{lem}
\begin{proof}
By changing coordinates we may assume that $\left\vert \b \right\vert \leq 1$ 
and that $\left\vert \gamma \right\vert \leq 1$ for every critical point $\gamma$ of $\vp$. 
Writing $\vp = f/g$ for relatively prime polynomials $f,g \in \mathfrak{o}_k \left[ x \right]$. 
The roots of the Wronskian of $\vp$ are precisely the critical points of $\vp$ and 
$$\text{Wr}_\vp \left(x \right) = g \left( x \right) f^\prime \left( x \right) - g^\prime \left( x \right) f \left( x \right)  
= a \prod_{\vp^\prime \left( \gamma \right) = 0} \left( x - \gamma \right),$$
for some $a \in \mathfrak{o}_k$. Since $\vp $ has good separable reduction, 
the reduction of the last expression is not identically zero and in particular $\left\vert a \right\vert = 1$.
Evaluating the reduction at $\overline{\b}$ shows that
$$0 = \text{Wr}_{\overline{\vp}}\left( \overline{\b} \right) = \overline{\text{Wr} \left( \b \right)} 
= \overline{a} \prod_{\vp^\prime \left( \gamma \right) = 0} \left( \overline{\b} - \overline{\gamma} \right).$$
It follows that $\overline{\b} = \overline{\gamma}$ for some critical point $\gamma$, proving the lemma.
\end{proof}

We now apply this local result to our global field $F$. 

\begin{lem}\label{RamPoints}
Let $\tilde{\p}$ be a prime of $\mathfrak{o}_{\overline{F}}$ (where $\overline{F}$ is an algebraic closure of $F$ and $\mathfrak{o}_{\overline{F}}$ is its ring of integers), 
and let $\vp = h \left( x^r \right) \in \Fx$ be a rational map of separable degree $d \geq 2$, where $h$ is separable and $r$ is the inseparable degree of $\vp$.
Let $\p = \tilde{\p}\bigcap \mathfrak{o}_F$ be a prime such that $h$ has good and separable reduction at $\p$. 
Also let $\alpha \in \mathbb{P}^1 \left( F \right)$,
$E$ the splitting field of $\vp \left( x \right) - \a$, 
$\beta \in \vp^{-1} \left( \alpha \right) \subseteq \mathbb{P}^1 \left( E \right)$, 
and $\p^\prime := \tilde{\p} \bigcap E$. 
If $\p^\prime$ is ramified over $\p$ then $\beta$ is congruent modulo $\tilde{\p}$ to a ramification point of $\vp^m$. 
\end{lem}
\begin{proof}
Let $\left\vert \cdot \right\vert_{\tilde{\p}}$ be the $\tilde{\p}$-adic absolute value on $\overline{F}$.
By changing coordinates we can assume that $\alpha = 0$ and $\left\vert \beta \right\vert_{\tilde{\p}} \leq 1$. 
Let $h \left( x \right) = \frac{f\left( x \right)}{g \left( x \right)}$ ($f, g \in F\left[ x \right]$). 
As $\vp \left( \beta \right) = \alpha = 0$ we have that $h \left( \b \right) = f \left( \beta  \right) = 0$. 
Since $\p^\prime $ is ramified over $\p$, by definition $\p$ is ramified in $\Eu$, the maximal separable sub-extension of $E$. 
Additionally, since $\p^\prime$ is ramified $f$ must have at least one other root congruent to $\beta$ modulo $\tilde{\p}$, thus $\overline{f}$ has a multiple root at $\beta$. 
Since $h$ has good reduction at $\p$, $\overline{g} \left( \beta \right) \neq 0$, 
therefore $\overline{h}$ has a multiple root at $\beta$, hence $\overline{h}^\prime \left( \beta \right) = 0$. 
Applying Lemma \ref{critical_point} finishes the proof. 
\end{proof}

Using Lemma \ref{RamPoints} we now show that there there is a field extension $L/F$ where we can find an unramified extension with the necessary extension of residue fields. 

\begin{lem}\label{Ramification}
Let $\tilde{\p}$ be a prime of $\mathfrak{o}_{\overline{F}}$ and let $\vp \in \Fx$ 
be a rational function with separable degree $d \geq 2$, and separable good reduction at $\p = \tilde{\p}\bigcap \mathfrak{o}_F$, 
and let $\b \in \mathbb{P}^1 \left( F \right)$. 
Then there exists a finite extension $E$ of $F$ with the following property: 
for any finite extension $L$ of $E$, there is a positive integer $M$ such that for any $m \geq M$ and all 
$\a \in \mathbb{P}^1 \left( \overline{F} \right)$ with $\vp^m \left( \a \right) = \b$, 

\begin{itemize}
\item[(i)] $\tau$ does not ramify over $\p^\prime$, and
\item[(ii)] $\left[ \mathfrak{o}_{L \left( \a \right)}/\tau : \mathfrak{o}_L / \p^\prime \right] > 1$
\end{itemize}
where $\tau := \tilde{\p} \bigcap \mathfrak{o}_{L \left( \a \right)}$ and $\p^\prime = \tilde{\p} \bigcap \mathfrak{o}_L$. 
\end{lem}

\begin{proof}
We begin by assuming that $\b$ is not periodic modulo $\p$. In this case for any $\gamma \in \mathfrak{o}_{\overline{F}}$, there is at most one $j \geq 0$ such that 
\begin{equation}\label{finite} \vp^j \left( \gamma \right) \equiv \b \mod \tilde{\p}. \end{equation}

In particular, for each ramification point $\gamma \in \mathbb{P}^1 \left( \overline{F} \right)  $ of $\vp$, 
there are only finitely many integers $n \geq 0 $ and points $z \in \mathbb{P}^1 \left( \overline{F} \right)$ 
such that $\vp^n \left( z \right) = \b$ and $z \equiv \gamma \mod \tilde{\p}$. 
Let $E$ be the finite extension of $F$ formed by adjoining all such points $z$. 

Given any finite extension $L$ of $E$, let $\p^\prime = \tilde{\p} \bigcap \mathfrak{o}_L$. 
Since $\mathbb{P}^1 \left( \mathfrak{o}_L / \p^\prime \right)$ is finite, 
(\ref{finite}) implies that for all sufficiently large $M$, 
$\vp^M \left( x \right) = \b$ has no solutions in $\mathbb{P}^1 \left( \mathfrak{o}_L / \p^\prime \right)$. 
Fix any such $M$; note that $M$ must be larger than any of the integers $n$ above. 
Hence, given any $m \geq M$ and $\a \in \mathbb{P}^1 \left( \overline{F} \right)$ such that $\vp^m \left( \a \right) = \b$, 
we must have $\left[ \mathfrak{o}_{L \left( \a \right)} / \tau : \mathfrak{o}_L / \p^\prime \right] > 1$, 
where $\tau = \tilde{\p} \bigcap \mathfrak{o}_{L \left( \a \right)}$, proving (ii). 
Furthermore, if $\a$ is a root of $\vp^m \left( x \right)  - \b$, then there are two possibilities: 
either (1) $\a$ is not congruent modulo $\tilde{\p}$ to a ramification point of $\vp^m$, or (2) $\vp^j \left( \a \right) = z $ 
for some $j \geq 0$ and some point $z \in \mathbb{P}^1 \left( L \right)$ from the previous paragraph. 
In case (1), $\tau$ is unramified over $\p^\prime$ by Lemma \ref{RamPoints}. 
In case (2), choosing a minimal such $j \geq 0 $, and applying Lemma \ref{RamPoints} with $z$ in the roll of $\b $ and $j$ in the role of $m$, 
$\tau$ is again unramified over $\p^\prime$. 
Thus (ii) holds.

If $\b$ is periodic modulo $\p$ we will assume that it is fixed and then apply the above prove to each of the 
$\gamma_i \in \vp^{-1} \left( \b \right) \setminus \left\lbrace \b \right\rbrace$ and produce a field $E_i$ with the stated properties. 
To do this we note that no $\gamma_i$ is periodic modulo $\tilde{\p}$ as $\gamma_i \not\equiv \b \mod \tilde{\p}$; 
otherwise the ramification index of $\vp$ at $\b$ would be greater modulo $\tilde{\p}$ than over $F$ which contradicts the hypothesis of good and separable reduction. 
As $\gamma_i \not\equiv \b \mod \tilde{\p}$ it must be strictly preperiodic. 

For each $E_i$ let $M_i$ be the constant as defined above, let $M = \displaystyle\max_i\left( M_i \right) + 1$ and let $E$ be the compositum of the $E_i$. 
For any $m \geq M$ and $\a \in \mathbb{P}^1 \left( \overline{F} \right)$ such that $\vp^m \left( \a \right) = \b$ but $\vp^t \left( \a \right) \neq \b$ for any $0 \leq t < m$
it follows that $\vp^{m-1 } \left( \b \right) \neq \a$. Thus $\vp^{m-1} \left( \b \right)$ is one of the non periodic $\gamma_i$. 
Applying the arguments above for the $\gamma_i$ we see that $\b$ satisfies conditions (i) and (ii). 
\end{proof}

We now have enough tools to prove Theorem \ref{ffintersection}. We begin by showing the result holds when $\a$ is not periodic. 

\begin{lem}\label{big}
Let $F$ be a function field with a finite field of constants. 
If $\vp \in F\left( x \right)$ and $\a,\b \in F$, where $\vp$ is a rational function with separable of degree at least 2, 
$\a$ is not periodic and $\b \notin \mathcal{O}_\vp^+ \left( \a \right)$, 
then there is a set of primes $\mathcal{P}$ of positive density such that for any $\p \in \mathcal{P}$ and $n \geq 1$ we have 
$$\vpn \left( \a \right)  \not\equiv \b \mod \p.$$
\end{lem}

\begin{proof} 
By Corollary \ref{notperiodic} there are infinitely many primes $\p$ of $F$ for which $\a$ has non-periodic reduction.  
As there are only finitely many primes of bad reduction and finitely many primes with inseparable reduction we can choose a prime $\tau$ of $F$ of good reduction for which $\a$ is not periodic.

By Lemma \ref{Ramification} there is a finite extension $E/F$ which satisfies the conclusions of the lemma. $E$ can be thought of as a finite extension of itself. By Lemma \ref{densityextension} it suffices to prove the result for $E$. Let $\p$ be a prime of $E$ extending $\tau$. For a sufficiently large integer $M$, and each $w \in \overline{F}$ such that	$\vp^m \left( w \right) = \b $ but $\vp^t \left( w \right) \neq \b$ for every $0 \leq t < M$ we have
\begin{itemize}
\item[(i)] $\p^\prime$ does not ramify over $\p$, and
\item[(ii)] $\left[ \mathfrak{o}_{E \left( w \right)} / \p^\prime : \mathfrak{o}_E / \p \right]>1$.
\end{itemize}
Where $\p^\prime$ is a prime of $E \left( w \right) $ extending $\p$. 

Fix such an $M$, and let $L/E$ be the splitting field for $\vp^M$, and note that this is a Galois extension as it is a splitting field. 
Also, by property (i) above $L/E$ is unramified over $\p$. 
By property (ii), the Frobenius element of $\p$ belongs to a conjugacy glass of $G:=\Gal\left( L/E \right)$ whose members do not fix any of the point $w$. 
By the Chebotarev Density Theorem, \ref{MurtyScherkCheb}, this implies that there is a set of primes, $\mathcal{P}$, with positive density whose Frobenius conjugacy class in $\Gal \left(L/F \right)$ fix none of the $w$.

Fix any prime $\pi \in \mathcal{P}$. Let $m \geq 0$ be an integer and $z \in E$ be a point such that $\vp^m \left( z \right) \equiv \b \mod \pi$. 
We claim that there is some $0 \leq n < M$ such that $\vp^n \left( z \right) \equiv \b \mod \pi$.

To begin the proof note that if $m <M $ we are already done, so assume that $m \geq M$. 
We can also assume that $m$ is the minimal integer $m \geq M$ such that $\vp^m \left( z \right) \equiv \b \mod \tau$. 
By the definition of $\mathcal{P}$, there can be no $c \in E$ such that $\vp^M \left(c \right) \equiv \b \mod \pi$ where $\vp^t \left( c \right) \not\equiv \b \mod \pi$ for every $0 \leq t < M$.  

Let $c = \vp^{m-M} \left( z \right) \in E$. 
As $m$ was chosen to be the minimal integer greater than $M$ there must be a $t$, with $0 \leq t < M$, 
such that $\vp^t \left( \vp^{m-M} \left( z \right) \right) \equiv \b \mod \pi$. 
So $\vp^{m-M+t} \left( z \right) \equiv \b \mod \pi$. 
But $0 \leq m-M +t < m$, which contradicts the minimality of $m$. 
Proving the claim. 

If $\b$ is preperiodic (including the case where $\b$ is periodic), 
then let $\mathcal{U} \subseteq \mathcal{P}$ be the set of primes $\pi$ such that $\vp^t \left( \a \right) \equiv \b \mod \pi$, 
for some $0 \leq t <M $. $\mathcal{U}$ is a finite set. 
Remove the primes of $\mathcal{U}$ from $\mathcal{P}$ and note that, as $\mathcal{U}$ is finite, the density of $\mathcal{P}$ has not changed.

If $\b$ is not preperiodic, then let $\mathcal{V} \subseteq \mathcal{P}$ be the set of primes $\pi$ such that $\vp^t \left( \b \right) \equiv \b \mod \pi$ for some $1 \leq t < M $. 
$\mathcal{V}$ is a finite set. 
Remove the primes of $\mathcal{V}$ from $\mathcal{P}$ and again note that, as $\mathcal{V}$ is finite, the density of $\mathcal{P}$ has not changed.

Suppose there is a $\pi \in \mathcal{P}$ and an integer $m \geq M$ such that $\vp^m \left( \a \right) \equiv \b \mod \pi$. 
Then by the earlier claim there is a $0 \leq t <M $ such that $\vp^t \left( \a \right) \equiv b \mod \pi$. 
So by the construction of $\mathcal{U}$ we must have that $\b$ is not preperiodic. 
Since $\vp^{m-t-1} \left( \vp \left( \b \right) \right) \equiv \b \mod \pi$, and because $m-t-1 \geq 0$, the claim tells us that there is a $0 \leq k < M$ such that $\vp^{k+1} \left( \b \right) \equiv \b \mod \pi$. 
But this is impossible by the construction of $\mathcal{V}$. Proving the lemma. 
\end{proof}

After applying Lemma \ref{big} all that remains to prove of Theorem \ref{ffintersection} is the case where $\alpha$ is periodic. 
We will now examine that case and conclude the proof. 

\begin{proof}[Proof of Theorem \ref{ffintersection}]
If $\a$ is not periodic then Theorem \ref{ffintersection} follows from Lemma \ref{big}.
If $\a$ is periodic then it has a finite orbit $\mathcal{O}_\vp \left( \a \right) = \left\lbrace \a = \a_0, \a_1, \dots, a_m \right\rbrace$. 
Since only finitely many primes contain the set of $\beta - a_i$ it follows that for any prime $\p$ outside of that finite set $\vpn \left( \alpha \right)  \not\equiv \beta \mod \p$ for every $n$. 
\end{proof}

\subsection{Stronger Non-Periodic Reduction}\label{subsec:consequences}
We now apply Theorem \ref{ffintersection} and Theorem \ref{nfintersection} to prove a stronger version of Corollary \ref{notperiodic}. Recall that Corollary \ref{notperiodic} says that if $F$ is a global field, $\vp \in \Fx$ is a rational function of degree at least 2, and $\a \in F$ wanders then there are infinitely many primes $\p$ for which $\a$ has non-periodic reduction. Theorem \ref{ffintersection} allows us to say something about the this infinite set of primes.

\begin{cor}\label{strong_non_periodic}
Let $F$ be a global field, $\vp \in \Fx$ a rational map of degree at least 2 and let $\a \in F$ be a non-periodic point. There exists a set of primes $S$ with positive density such that $\a$ does not have periodic reduction for any prime $\p \in S$.
\end{cor}

\begin{proof}
Since $\a$ is not periodic, by definition $\a \notin \mathcal{O}_\vp^+ \left( \a \right)$. Thus by Theorems \ref{ffintersection} and \ref{nfintersection} there is a set of primes $\mathcal{P}$ with positive density for which $\vpn \left( \a \right) \not\equiv \a \mod \p$ for any $\p \in \mathcal{P}$ and every $n \geq 1$. Thus $\a$ does not have periodic reduction for any $\p \in S$.
\end{proof}

\section{An Expansion Theorem 3}\label{sec:Generalized}

In this section we expand Theorem 3 in the style of \cite{BGHKST} to include sets of points instead of a single $\a$ and $\b$. 

\begin{thm*}\label{Thm_Generalized}
 Let $F$ be a function field with a finite field of constants and let $\vp_1, \dots, \vp_g: \mathbb{P}^1\left(F \right) \rightarrow \mathbb{P}^1 \left( F \right)$ be a set of rational maps each with separable degree at least 2. 
 Let $\mathcal{A}_1, \dots, \mathcal{A}_g$ be finite subsets of $\mathbb{P}^1 \left( F \right) $ such that no $\mathcal{A}_i$ contain a $\vp_i$-preperiodic point and let $\mathcal{B}_1, \dots, \mathcal{B}_g$ be finite subsets of $\mathbb{P}^1_F$ such that 
 at most one $\mathcal{B}_i$ contains a point which is not $\vp_i$-preperiodic, and that there is at most one such point in that set.
 There is a set of primes $\mathcal{P}$ of $F$ with positive density and a positive integer $M$ such that for any $i = 1, \dots, g$, any $\a \in \mathcal{A}_i$, any $\b \in \mathcal{B}_i$, any $\p \in \mathcal{P}$, and every $n \geq M$,
 $$\vp^n \left( \a \right) \not\equiv \b \mod \p.$$
\end{thm*}

\begin{rmk}
 Theorem \ref{Thm_Generalized} is both stronger and slightly weaker than Theorem \ref{ffintersection}. 
 Theorem \ref{Thm_Generalized} is stronger because it allows for sets of points and it makes no requirement about the intersection of orbits in $F$.
 It is weaker because it does not allow any of the $\a$ to be preperiodic, as Theorem \ref{ffintersection} does, and, since the wandering $\b$ may be in the orbit of an $\a$, we are not able to
 remove the possiblity that $M >0$. 
\end{rmk}

\begin{proof}
 The proof of Theorem 3.1 in \cite{BGHKST} works exactly for proving Theorem \ref{Thm_Generalized} if one replaces their use of 
 Lemma 4.3 from \cite{BGKT} with Theorem \ref{notperiodic} and
 replaces their Proposition 3.4 with Proposition \ref{Prop:Ramification_Sets} below.

\end{proof}

The following Proposition applies Lemma \ref{Ramification} to a set of points $\mathcal{B}$.

\begin{prop}\label{Prop:Ramification_Sets}
 Let $F$ be a function field with a finite field of constants, let $\tilde{\p}$ be a prime of $\mathfrak{o}_{\overline{F}}$, and let $\vp:\mathbb{P}^1 \rightarrow \mathbb{P}^1$ be a rational function
 defined over $F$ with separable degree at least 2 and with separable good reduction at $\p = \tilde{\p} \bigcap \mathfrak{o}_F$.
 Let $\mathcal{B} = \left\lbrace \b_1, \b_2, \dots, \b_n \right\rbrace \subset \mathbb{P}^1 \left( F \right)$ be such that for each $\b_i \in \mathcal{B}$ either
 \begin{itemize}
  \item if $\b_i$ is not periodic, then $\b_i$ is not periodic modulo $\p$; and
  \item if $\b_i$ is periodic, then $\vp \left( \b_i \right) = \b_i $ and the ramification index of $\vp$ at $\b_i$ is the same modulo $\p$ as over $F$.
 \end{itemize}
 
 There exists a finite extension $E $ of $F$ with the following property: for any extension $L$ of $E$, there is an integer $M \geq 0$ such that for all $m \geq M$
 and all $\a \in \mathbb{P}^1 \left( F \right)$ with $\vp^m \left( \a \right) \in \mathcal{B}$ but $\vp^t \left( \a \right) \notin \mathcal{B}$ for any $t < m$,
 \begin{enumerate}
  \item $\tau$ does not ramify over $\q$, and 
  \item $\left[\mathfrak{o}_{L \left( \a \right)} / \tau : \mathfrak{o}_L / \q \right] >1$,
 \end{enumerate}
where $\tau := \tilde{\p} \bigcap \mathfrak{o}_{L \left( \a \right)}$ and $\q := \tilde{\p} \bigcap \mathfrak{o}_L$.

\end{prop}

\begin{proof}
 The proof is the same as that of Proposition 3.4 in \cite{BGHKST}. 
 The requirement of separable degree of at least 2 for $\vp$ is in place so that we can apply Lemma \ref{Ramification} in place of 
 their Lemma 3.3. 

\end{proof}

\section{A Hasse Principle for Periodic Points}\label{sec:Hasse}

We now complete the proof of Theorem \ref{hasse}. 

\begin{proof}[Proof of Theorem \ref{hasse}]
It is left to show that if $\a \in F$ has periodic reduction for every prime $\p \in \mathcal{P}$ where $\mathcal{P}$ is a set of primes with positive natural density 1 then $\a $ is periodic in $F$. We will prove the contrapositive.

If $\a$ is not periodic then by Corollary \ref{strong_non_periodic} there is a set of primes $S$ with positive density for which $\a$ has non-periodic reduction. Thus it is impossible for $\a$ to have periodic reduction on a set of primes $\mathcal{P}$ with density 1.
\end{proof}

One might ask if the set $\mathcal{P}$ of primes with density 1 can be replaced by a set with a slightly smaller density? As stated the answer is no. Consider the family of polynomial maps $\vp_q: \mathbb{P}^1_\mathbb{Q} \rightarrow \mathbb{P}^1_\mathbb{Q}$ given by $\vp_q \left( x \right) = x^q+1$, where $q \in \mathbb{Z}$ is an odd prime. Let $\mathcal{P}_q$ be the set of primes $\mathcal{P}_q = \left\lbrace p \in \mathbb{Z} : p \not\equiv 1 \mod q  \right\rbrace$, and note that $\mathcal{P}_q$ has natural density $\dfrac{q-2}{q-1}$. 

For any prime $p \in \mathcal{P}_q$ the residue field $\mathbb{F}_p$ does not contain a $q$th root of unity and therefore $\vp_q \left(x \right) = x^q + 1$ is injective on the finite set $\mathbb{F}_p$. This means that $\overline{\vp}_q \left( x \right)$ is a permutation of $\mathbb{F}_p$. When a permutation of a set is iterated every point is periodic, in particular $\a = 0$ is periodic under the map $\overline{\vp_q } \left( x \right)$ for every prime $p \in \mathcal{P}_q$. In other words 0 has a periodic reduction for every prime $p \in \mathcal{P}_q$, but 0 is a wandering point of $\vp_q \left( x \right)$ in $\mathbb{Q}$. Taking $q$ to be sufficiently large we produce a set of primes, $\mathcal{P}_q$, with density arbitrarily close to 1 and a map $\vp_q \left( x \right)$ for which $\a = 0$ wanders but has periodic reduction on $\mathcal{P}$. 

We have demonstrated a wandering points with non-periodic reduction only on sets of arbitrarily small positive density, but we did it by using maps of very large degree. One could hope to recover a weaker statement by bounding the degree of the map $\vpx$. So we are left with the following question.

\begin{question}\label{q1}
For a global field $F$, and a rational map $\vp \in \Fx$ of degree $d$, is there a constant $C$ (depending on $d$) such that for any set of primes $\mathcal{P}$ with density $D$ satisfying $1 - D < C$ the following holds: If $\a \in F$ has periodic reduction for every $\p \in \mathcal{P}$ then must $\a$ be periodic in $F$?
\end{question}

We conclude with a heuristic as why such a $C$ should exist and propose a value. As usual let $F$ be a global field, and $\vp \in \Fx$ a rational map of degree $d \geq 2$. A wandering point $\a \in F$ will have periodic reduction on a set of primes $\mathcal{P}$ with a large density if that periodic reduction happens for some reason other than $\a$ being periodic, such as $\overline{\vp}$ inducing a permutation on the residue field at every $\p \in \mathcal{P}$. For a $F$ is a number field recall that for any prime $\p$ we denote by $F_\p$ the residue field of $\p$. If $F$ is a number field then Schur made the following conjecture in \cite{Schur}; if $\vpx \in F \left[x \right]$ is a polynomial which is bijective on $F_\p$ for infinitely many $\p$ then either $\vpx = ax^n + c$ or $\vpx = T_n \left( x \right)$ (the $n$-th Chebychev polynomial). Fried proved Schur's conjecture in \cite{Fried} and Guralnick, M{\"u}ller, and Saxl generalized the result to rational functions in \cite{GMS}, however by allowing all rational maps the possible types of maps which induce a permutation on infinitely many $F_\p$ must be expanded. These results imply that maps of the form $\vpx = x^d +c$ are in a sense `worst-possible'. However for each map in this family there is only one `bad' Chebotarev class of primes, the primes which are 1 modulo $q$. The density of these `bad' primes is $\dfrac{1}{ \phi \left( d \right)}$ where $\phi$ is the Euler $\phi$-function. We therefore propose that $C = \dfrac{1}{ \phi \left( d \right)}$ is the value which will give an affirmative answer to Question \ref{q1}.

\bibliographystyle{plain}
\bibliography{bib}

\begin{thebibliography}{10}

\bibitem{BGHKST}
Robert Benedetto, Dragos Ghioca, Benjamin Hutz, Pär Kurlberg, Thomas Scanlon,
  and Thomas Tucker.
\newblock Periods of rational maps modulo primes.
\newblock {\em Mathematische Annalen}, pages 1--24.
\newblock 10.1007/s00208-012-0799-8.

\bibitem{BGKT}
Robert~L. Benedetto, Dragos Ghioca, P{\"a}r Kurlberg, and Thomas~J. Tucker.
\newblock A case of the dynamical {M}ordell-{L}ang conjecture.
\newblock {\em Math. Ann.}, 352(1):1--26, 2012.

\bibitem{BombieriRunge}
Enrico Bombieri.
\newblock On {W}eil's ``th\'eor\`eme de d\'ecomposition''.
\newblock {\em Amer. J. Math.}, 105(2):295--308, 1983.

\bibitem{Bombieri}
Enrico Bombieri and Walter Gubler.
\newblock {\em Heights in {D}iophantine geometry}, volume~4 of {\em New
  Mathematical Monographs}.
\newblock Cambridge University Press, Cambridge, 2006.

\bibitem{CSTZ}
C.~Corvaja, V.~Sookdeo, T.~Tucker, and U.~Zannier.
\newblock Integral points in two-parameter orbits.
\newblock {\em arXiv:1201.1313v2 [math.NT]}.

\bibitem{Fried}
Michael Fried.
\newblock On a conjecture of {S}chur.
\newblock {\em Michigan Math. J.}, 17:41--55, 1970.

\bibitem{GMS}
Robert~M. Guralnick, Peter M{\"u}ller, and Jan Saxl.
\newblock The rational function analogue of a question of {S}chur and
  exceptionality of permutation representations.
\newblock {\em Mem. Amer. Math. Soc.}, 162(773):viii+79, 2003.

\bibitem{MurtyScherk}
Vijaya Kumar~Murty and John Scherk.
\newblock Effective versions of the {C}hebotarev density theorem for function
  fields.
\newblock {\em C. R. Acad. Sci. Paris S\'er. I Math.}, 319(6):523--528, 1994.

\bibitem{Levin}
Aaron Levin.
\newblock Variations on a theme of {R}unge: effective determination of integral
  points on certain varieties.
\newblock {\em J. Th\'eor. Nombres Bordeaux}, 20(2):385--417, 2008.

\bibitem{Runge}
Carl Runge.
\newblock Uber das produkt funf aufeinander folgender zahlen in einer
  arithmetischen reihe.
\newblock {\em J. Reine Angew. Math.}, 100:425--435, 1887.

\bibitem{Schur}
Issai Schur.
\newblock {\"U}ber den zusammenhang zwischen einem problem der zahlentheorie
  und einem satz über algebraische funktionen.
\newblock {\em S.-B. Preuss. Akad. Wiss. Berlin}, pages 123--134, 1923.

\bibitem{Silverman93}
Joseph~H. Silverman.
\newblock Integer points, {D}iophantine approximation, and iteration of
  rational maps.
\newblock {\em Duke Math. J.}, 71(3):793--829, 1993.

\bibitem{SilvermanDynamicsBook}
Joseph~H. Silverman.
\newblock {\em The arithmetic of dynamical systems}, volume 241 of {\em
  Graduate Texts in Mathematics}.
\newblock Springer, New York, 2007.

\end{thebibliography}

\end{document}